\documentclass[12pt,reqno,pdflatex]{amsart}
\usepackage{graphicx,amsmath,amsthm}
\usepackage{bbm}

\pdfoutput=1

\newcommand{\binomial}[2]{\left( \begin{array}{c} {#1} \\
                        {#2} \end{array} \right)}

\newtheorem{theorem}{Theorem}[section]
\newtheorem{prop}[theorem]{Proposition}
\newtheorem{lemma}[theorem]{Lemma}

\theoremstyle{remark}

\theoremstyle{definition}

\def\K{K} 
\def\line{\mathcal{L}} 
\def\uline{\mathrm{line\,}} 
\def\outdeg{\mathrm{outdeg}}
\def\indeg{\mathrm{indeg}}
\def\iindeg{\mathrm{indeg}}

\def\source{\mathtt{s}}
\def\target{\mathtt{t}}
\def\ssource{\mathtt{s}}
\def\ttarget{\mathtt{t}}
\def\root{\mathrm{root}}

\def\Im{\mathrm{Im}}
\def\Id{\mathrm{Id}}
\def\xx{\mathbf{x}}

\def\Kautz{\mathrm{Kautz}}
\def\Sylow{\mathrm{Sylow}}

\DeclareSymbolFont{AMSb}{U}{msb}{m}{n}
\DeclareMathSymbol{\C}{\mathbin}{AMSb}{"43} 
\DeclareMathSymbol{\EE}{\mathbin}{AMSb}{"45} 
\DeclareMathSymbol{\N}{\mathbin}{AMSb}{"4E} 
\DeclareMathSymbol{\PP}{\mathbin}{AMSb}{"50} 
\DeclareMathSymbol{\Q}{\mathbin}{AMSb}{"51} 
\DeclareMathSymbol{\R}{\mathbin}{AMSb}{"52} 
\DeclareMathSymbol{\Z}{\mathbin}{AMSb}{"5A}

\begin{document}

\title[Sandpile Groups and Spanning Trees]{Sandpile Groups and Spanning Trees of Directed Line Graphs}
\author{Lionel Levine}

\address{Lionel Levine, Department of Mathematics, Massachusetts Institute of Technology, Cambridge, MA 02139. {\tt \url{http://math.mit.edu/~levine}}}
\thanks{The author is supported by an NSF postdoctoral fellowship.}

\date{March 24, 2010}
 
\begin{abstract}
We generalize a theorem of Knuth relating the oriented spanning trees of a directed graph~$G$ and its directed line graph~$\line G$.
The sandpile group is an abelian group associated to a directed graph, whose order is the number of oriented spanning trees rooted at a fixed vertex.
In the case when~$G$ is regular of degree~$k$, we show that the sandpile group of~$G$ is isomorphic to the quotient of the sandpile group of $\line G$ by its $k$-torsion subgroup.
As a corollary we compute the sandpile groups of two families of graphs widely studied in computer science, the de~Bruijn graphs and Kautz graphs.  
\end{abstract}

\maketitle

\section{Introduction}

Let $G=(V,E)$ be a finite directed graph, which may have loops and multiple edges.  Each edge $e \in E$ is directed from its source vertex $\source(e)$ to its target vertex $\target(e)$. The \emph{directed line graph} $\line G = (E,E_2)$ has as vertices the edges of $G$, and as edges the set
	\[ E_2 = \{ (e_1,e_2) \in E\times E \,|\, \source(e_2)=\target(e_1) \}. \]
For example, if $G$ has just one vertex and~$n$ loops, then $\line G$ is the complete directed graph on $n$ vertices (which includes a loop at each vertex).  If $G$ has two vertices and no loops, then $\line G$ is a bidirected complete bipartite graph.

An \emph{oriented spanning tree} of $G$ is a subgraph containing all of the vertices of~$G$, having no directed cycles, in which one vertex, the \emph{root}, has outdegree~$0$, and every other vertex has outdegree~$1$.  
The number~$\kappa(G)$ of oriented spanning trees of~$G$ is sometimes called the \emph{complexity} of~$G$.

Our first result relates the numbers $\kappa(\line G)$ and $\kappa(G)$.  Let $\{x_e\}_{e\in E}$ and $\{x_v\}_{v\in V}$ be indeterminates, and consider the polynomials 
\pagebreak
	\[ \kappa^{edge}(G,\xx) = \sum_T \prod_{e \in T} x_e \]
	\[ \kappa^{vertex}(G,\xx) = \sum_T \prod_{e \in T} x_{\ttarget(e)}. \] 
The sums are over all oriented spanning trees~$T$ of~$G$.

Write 
	\[ \indeg(v) = \# \{e \in E \,|\, \target(e)=v\} \]
	\[ \outdeg(v) = \# \{e \in E \,|\, \source(e)=v\} \]
for the indegree and outdegree of vertex~$v$ in~$G$.  We say that~$v$ is a \emph{source} if $\indeg(v)=0$.  

\begin{theorem}
\label{thm:weightedtreeenum}
Let $G=(V,E)$ be a finite directed graph with no sources.  Then
	\begin{equation} \label{weightedtreeenum} \kappa^{vertex}(\line G,\xx) = \kappa^{edge}(G,\xx) \prod_{v \in V} \left( \sum_{\ssource(e)=v} x_e \right)^{\mbox{\em \scriptsize indeg}(v)-1}. \end{equation}
\end{theorem}

Note that since the vertex set of $\line G$ coincides with the edge set of $G$, both sides of (\ref{weightedtreeenum}) are polynomials in the same set of variables $\{x_e\}_{e\in E}$.
Setting all $x_e=1$ 
yields the product formula
	\begin{equation}
	\label{outtothein} 
	\kappa(\line G) = \kappa(G) \prod_{v \in V} \outdeg(v)^{\iindeg(v)-1}
	\end{equation}
due in a slightly different form to Knuth \cite{Knuth67}.
Special cases of (\ref{outtothein}) include Cayley's formula $n^{n-1}$ for the number of rooted spanning trees of the complete graph $K_n$, as well as the formula $(m+n)m^{n-1}n^{m-1}$ for the number of rooted spanning trees of the complete bipartite graph $K_{m,n}$.  
These are respectively the cases that~$G$ has just one vertex with~$n$ loops, or~$G$ has just two vertices~$a$ and~$b$ with ~$m$ edges directed from~$a$ to~$b$ and~$n$ edges directed from~$b$ to~$a$.  

Suppose now that $G$ is \emph{strongly connected}, that is, for any~$v,w \in V$ there are directed paths in~$G$ from~$v$ to~$w$ and from~$w$ to~$v$.  Then associated to any vertex~$v_*$ of~$G$ 
is an abelian group $K(G,v_*)$, the \emph{sandpile~group}, whose order is the number 
of oriented spanning trees of $G$ rooted at~$v_*$.  Its definition and basic properties are reviewed in section~\ref{sandpilegroups}.  Other common names for this group are the critical group, Picard group, Jacobian, and group of components.  In the case when $G$ is \emph{Eulerian} (that is, $\indeg(v)=\outdeg(v)$ for all vertices $v$) the groups $K(G,v_*)$ and $K(G,v'_*)$ are isomorphic for any $v_*,v'_* \in V$, and we often denote the sandpile group just by $K(G)$.

When $G$ is Eulerian, we show that
there is a natural map from the sandpile group of $\line G$ to the sandpile group of $G$, 
descending from the $\Z$-linear map
	\[ \phi : \Z^E \to \Z^V \]
which sends $e \mapsto \target(e)$.

Let $k$ be a positive integer.  We say that $G$ is \emph{balanced $k$-regular} if $\indeg(v)=\outdeg(v)=k$ for every vertex~$v$. 	

\begin{theorem}
\label{mainsequence}
Let $G=(V,E)$ be a strongly connected Eulerian directed graph, fix $e_* \in E$ and let $v_* = \target(e_*)$.  The map $\phi$ descends to a surjective group homomorphism 
	\[ \bar{\phi}: \K(\line G,e_*) \to \K(G,v_*).  \] 
Moreover, if $G$ is balanced $k$-regular, then $\ker(\bar{\phi})$ is the $k$-torsion subgroup of $\K(\line G,e_*)$.
\end{theorem}

This result extends to directed graphs some of the recent work of Berget, Manion, Maxwell, Potechin and Reiner~\cite{linegraphs} on undirected line graphs.  If $G=(V,E)$ is an undirected graph, the (undirected) line graph $\mbox{line}(G)$ of $G$ has vertex set $E$ and edge set
	\[ \{ \{e,e' \} \,|\, e,e' \in E, ~e \cap e' \neq \emptyset \}. \]
The results of~\cite{linegraphs} relate the sandpile groups of $G$ and $\mbox{line}(G)$. The undirected case is considerably more subtle, because although there is still a natural map $\K (\uline G) \to \K (G)$ when $G$ is regular, this map may fail to be surjective.

A particularly interesting family of directed line graphs are the \emph{de~Bruijn graphs} $DB_n$, defined recursively by
	\[ DB_n = \line (DB_{n-1}), \qquad n\geq 1, \]
where $DB_0$ is the graph with just one vertex and two loops.  The~$2^n$ vertices of $DB_n$ can be identified with binary words $b_1 \ldots b_n$ of length~$n$; two such sequences~$b$ and~$b'$ are joined by a directed edge $(b,b')$ if and only if $b'_i=b_{i+1}$ for all $i=1,\ldots,n-1$.

Using Theorem~\ref{mainsequence}, we obtain the full structure of the sandpile groups of the de Bruijn graphs. 

\begin{theorem}
	\label{deBruijn}
	\[ \K(DB_n) = \bigoplus_{j=1}^{n-1} \,(\Z/2^j\Z)^{2^{n-1-j}}. \]
\end{theorem}

Closely related to the de Bruijn graphs are the \emph{Kautz graphs}, defined by
	\[ \Kautz_1 = (\{1,2,3\}, \{(1,2),(1,3),(2,1),(2,3),(3,1),(3,2)\}) \]
and 
	\[ \Kautz_n = \line (\Kautz_{n-1}), \qquad n \geq 2. \]
The Kautz graphs are useful in network design because they have close to the maximum possible number of vertices given their diameter and degree~\cite{FYM84} and because they contain many short vertex-disjoint paths between any pair of vertices~\cite{DLH93}.  The following result gives the sandpile group of $\Kautz_n$.

\begin{theorem}
	\label{Kautz}
	\[ \K(\mbox{\em Kautz}_n) = (\Z/3\Z) \oplus (\Z/2^{n-1}\Z)^2 \oplus \bigoplus_{j=1}^{n-2} (\Z/2^j\Z)^{3\cdot 2^{n-2-j}}. \]
\end{theorem}

The remainder of the paper is organized as follows.  In section~\ref{spanningtrees}, we prove Theorem~\ref{thm:weightedtreeenum} and state a variant enumerating spanning trees with a fixed root.  Section~\ref{sandpilegroups} begins by defining the sandpile group, and moves on from there to the proof of Theorem~\ref{mainsequence}.  
In section~\ref{iteratedlinegraphs} we 
enumerate spanning trees of iterated line digraphs.  Huaxiao, Fuji and Qiongxiang \cite{iterated} prove that for a balanced $k$-regular directed graph~$G$ on~$N$ vertices,
	\[ \kappa(\line^n G) = \kappa(G) k^{(k^n-1)N}. \]
Theorem~\ref{iterated} generalizes this formula to an arbitrary directed graph~$G$ having no sources.  This section also contains the proofs of Theorems~\ref{deBruijn} and~\ref{Kautz}.  Lastly, in section~\ref{concluding} we pose two questions for future study.

\section{Spanning Trees}
\label{spanningtrees}

Let $G=(V,E)$ be a finite directed graph, loops and multiple edges allowed.  We denote its vertices by $v,w,\ldots$ and edges by $e,f,\ldots$.  Each edge $e \in E$ is directed from its source $\source(e)$ to its target $\target(e)$.    In this section we prove Theorem~\ref{thm:weightedtreeenum} relating the spanning trees of~$G$ and~$\line G$, and discuss some interesting special cases. 

If~$k$ is a field, we write~$k^V$ and~$k^E$ for the $k$-vector spaces with bases indexed by~$V$ and~$E$ respectively.  We think of the elements of~$k^V$ or~$k^E$ as formal $k$-linear combinations of vertices or of edges.

Consider the field of rational functions $\Q(\xx) = \Q((x_e)_{e\in E}, (x_v)_{v \in V})$.  
The \emph{edge-weighted Laplacian} and \emph{vertex-weighted Laplacian} of $G$ are the $\Q(\xx)$-linear transformations 
	\[	\Delta^{edge}, \Delta^{vertex} : \Q(\xx)^V \to \Q(\xx)^V \]
%
sending 
	\[ \begin{split} 
		\Delta^{edge}(v) &= \sum_{\ssource(e)=v} x_e (\target(e) - v); \\
	 	\Delta^{vertex}(v) &= \sum_{\ssource(e)=v} x_{\ttarget(e)} (\target(e) - v).
	\end{split} \]
The sums are over all edges~$e\in E$ such that $\source(e)=v$.

We will use the following form of the matrix-tree theorem for directed graphs.  Here $[t]\,p(t)$ denotes the coefficient of $t$ in the polynomial $p(t)$.

\begin{theorem}[Matrix-Tree Theorem]
	\[ \kappa^{edge}(G,\xx) = [t] \det(t \cdot \mbox{\em Id} - \Delta^{edge}). \]
	\[ \kappa^{vertex}(G,\xx) = [t] \det(t \cdot \mbox{\em Id} - \Delta^{vertex}). \]
\end{theorem}

For a proof, see for example \cite[Theorem~2]{CL96} for the vertex weighted-version, and \cite{Chaiken82} for the edge-weighted version.

\begin{proof}[Proof of Theorem~\ref{thm:weightedtreeenum}]
Consider the $V\times E$ matrix
	\[ A_{ve} = \begin{cases} 1, & v=\target(e) \\ 0, & \mbox{else}. \end{cases} \]
and the $E\times V$ matrix
	\[ B_{ev} = \begin{cases} x_e, & v=\source(e) \\ 0, & \mbox{else}. \end{cases} \]

Let $\Delta$ be the edge-weighted Laplacian of $G$, and let $\Delta^\line$ be the vertex-weighted Laplacian of $\line G$.  Then
	\[  \Delta = AB - D \]
and
	\begin{equation} \label{biglaplacian} \Delta^{\line} = BA - D^\line \end{equation}
where $D$ and $D^\line$ are the diagonal matrices with diagonal entries
	\[ D_{vv} = \sum_{\ssource(f)=v} x_f, \qquad v \in V \]
and
	\[ D^{\line}_{ee} = \sum_{\ssource(f)=\ttarget(e)} x_f, \qquad e \in E. \]

Since $AD^\line = DA$, we have
	\begin{equation} \label{intertwining} A\Delta^\line = A(BA-D^\line) = ABA - DA = (AB-D)A = \Delta A. \end{equation}
In particular, $\Delta^\line(\ker(A)) \subset \ker(A)$, so the vector space decomposition
	\[ \Q(\xx)^E = \ker(A) \oplus \ker(A)^\perp \]
exhibits $\Delta^\line$ in block triangular form.  Hence the characteristic polynomial $\chi(t)$ of $\Delta^\line$ factors as
	\[ \chi(t) = \chi_1(t) \chi_2(t) \]
where $\chi_1$ and $\chi_2$ are respectively the characteristic polynomials of $\Delta^\line|_{\ker(A)}$ and $\Delta^\line|_{\ker(A)^\perp}$.  

By hypothesis, $G$ has no sources, so $A$ has full rank.  In particular, $AA^T$ is invertible.  Hence the restriction $A|_{\ker(A)^\perp}$ is an isomorphism of $\ker(A)^\perp = \Im(A^T)$ onto $\Q(\xx)^V$.  By (\ref{intertwining}) it follows that $\Delta^\line|_{\ker(A)^\perp}$ and~$\Delta$ have the same characteristic polynomial
	\[ \chi_2(t) = \det(t\cdot\Id - \Delta). \]
Since the rows of $\Delta$ sum to zero, $\chi_2(t)$ has no constant term.  By the matrix-tree theorem,
	\begin{equation*} \label{coefficientoft} \begin{split} \kappa^{vertex}(\line G,\xx) = [t]\chi(t)
		&= \chi_1(0) \cdot [t] \chi_2(t) \\
		&= \det \left(-\Delta^\line|_{\ker(A)}\right) \cdot \kappa^{edge}(G,\xx).
	\end{split} \end{equation*}

It remains to find the determinant of $-\Delta^\line|_{\ker(A)}$.  For each vertex $v \in V$, fix an edge $e_0(v)$ with $\target(e_0(v))=v$.  Then a basis for $\ker(A)$ is given by the vectors
	\[ \alpha_e = e - e_0(v), \qquad v\in V,~e \in E,~\target(e)=v,~e\neq e_0(v). \]
By (\ref{biglaplacian}) we have
	\[ \Delta^\line \alpha_e = -\left(\sum_{\ssource(f)=\ttarget(e)} x_f \right) \alpha_e \]
so the vectors $\alpha_e$ form an eigenbasis for $\Delta^\line|_{\ker(A)}$.  As each eigenvalue $-\sum_{\ssource(f)=v} x_f$ occurs with multiplicity $\indeg(v)-1$, we conclude that
	\[ \det \left(-\Delta^\line|_{\ker(A)}\right) = \prod_{v \in V} \left( \sum_{\ssource(f)=v} x_f \right)^{\iindeg(v)-1}. \qed \]
\renewcommand{\qedsymbol}{}
\end{proof}

We remark that the idea of using the incidence matrices $A$ and $B$ to relate the adjacency matrices of $G$ and $\line G$ has appeared before.  See, for example, Yan and Zhang \cite[Proposition~1.4]{YZ03}, who in turn cite Lin and Zhang \cite{LZ83} and Liu \cite{Liu96}.

Theorem~\ref{thm:weightedtreeenum} enumerates all oriented spanning trees of $\line G$, while in many applications one wants to enumerate spanning trees with a fixed root.  Given a vertex $v_* \in V$, let	
	\[ \kappa^{edge}(G,v_*,\xx) = \sum_{\root(T)=v_*} \, \prod_{e \in T} x_e \]
and
	\[ \kappa^{vertex}(G,v_*,\xx) = \sum_{\root(T)=v_*} \, \prod_{e \in T} x_{\target(e)}. 
	\]
We will use the following variant of the matrix-tree theorem; see \cite{Chaiken82} and \cite[Theorem~5.6.4]{Stanley}.

\begin{theorem}[Matrix-Tree Theorem, rooted version]
\label{thm:rootedmatrixtree}
Let $\Delta^{edge}_0$ and $\Delta^{vertex}_0$ be the submatrices of $\Delta^{edge}$ and $\Delta^{vertex}$ omitting row and column~$v_*$.  Then
	\[ \kappa^{edge}(G,v_*,\xx) = \det (-\Delta^{edge}_0). \]
	\[ \kappa^{vertex}(G,v_*,\xx) = \det (-\Delta^{vertex}_0). \]
\end{theorem}
	
The following variant of Theorem~\ref{thm:weightedtreeenum} enumerates spanning trees of~$\line G$ with a fixed root $e_*$ in terms of spanning trees of~$G$ with root $w_*=\source(e_*)$.

\begin{theorem}
\label{thm:rootedtreeenum}
Let $G=(V,E)$ be a finite directed graph, and let $e_*=(w_*,v_*)$ be an edge of~$G$. If $\mbox{\em indeg}(v) \geq 1$ for all vertices $v \in V$, and $\mbox{\em indeg}(v_*) \geq 2$, then
	\begin{equation*} \frac{\kappa^{vertex}(\line G,e_*,\xx)}{x_{e_*} \kappa^{edge}(G,w_*,\xx)} =  \left( \sum_{\ssource(e)=v_*} x_e \right)^{\mbox{\em \scriptsize indeg}(v_*)-2}
	\prod_{v \neq v_*} \left( \sum_{\ssource(e)=v} x_e \right)^{\mbox{\em \scriptsize indeg}(v)-1}. \end{equation*}
\end{theorem}

\begin{proof}
The proof is analogous to that of Theorem~\ref{thm:weightedtreeenum}, except that it uses reduced incidence matrices
	\[ A_0 : \Q(\xx)^{E-\{e_*\}} \to \Q(\xx)^V \]
and
	\[ B_0 : \Q(\xx)^V \to \Q(\xx)^{E-\{e_*\}}. \]
%
%
The edge-weighted Laplacian of the graph $G\setminus e_* = (V,E-\{e_*\})$ is given by
	\[ \Delta_{G\setminus e_*} = A_0 B_0 - D + M \]
where the matrix~$M$ has a single nonzero entry~$x_{e_*}$ in row and column~$w_*$.  Expanding $\det(D - A_0 B_0)$ along column~$w_*$ we find
	\[ \det (D - A_0 B_0) = \det(-\Delta_{G\setminus e}) + x_{e_*} \det(-\Delta_0) \]
where $\Delta_0$ is the submatrix of the edge-weighted Laplacian of~$G$ omitting the row and column~$w_*$.  By Theorem~\ref{thm:rootedmatrixtree} we have $\det (-\Delta_0) = \kappa^{edge}(G,w_*,\xx)$.  Since the rows of $\Delta_{G\setminus e_*}$ sum to zero, it follows that
	\[ \det (D - A_0 B_0) = x_{e_*} \kappa^{edge}(G,w_*,\xx). \]

The submatrix $\Delta_0^\line$ of the vertex-weighted Laplacian of $\line G$ omitting the row and column~$e_*$ equals $B_0 A_0 - D_0^\line$, where $D_0^\line$ is the submatrix of $D^\line$ omitting row and column~$e_*$.  Since $A_0 D_0^\line = D A_0$, we have
	\[ A_0 \Delta_0^\line = A_0 (B_0 A_0 - D_0^\line) = A_0 B_0 A_0 - DA_0 = (A_0 B_0-D) A_0 \]
hence $\Delta_0^\line(\ker(A_0)) \subset \ker(A_0)$.  Now by Theorem~\ref{thm:rootedmatrixtree},
	\begin{align*} 
	\kappa^{vertex}(\line G, e_*, \xx)
		&= \det \left(-\Delta_0^\line \right) \\
		&= \det \left(-\Delta_0^\line|_{\ker(A_0)} \right) \det \left(-\Delta_0^\line|_{\ker(A_0)^\perp} \right). 
	\end{align*}
	
By hypothesis, the graph $G\setminus e_*$ has no sources, so $A_0$ has full rank.  The rest of the proof proceeds as before, giving
	\[ \det \left(-\Delta_0^\line|_{\ker(A_0)^\perp} \right) = \det(D - A_0 B_0) = x_{e_*} \kappa^{edge}(G,w_*,\xx) \]
and 
	\[ \det \left(-\Delta_0^\line|_{\ker(A_0)} \right) = \left( \sum_{\ssource(e)=v_*} x_e \right)^{\iindeg(v_*)-2} 
	\prod_{v \neq v_*} \left( \sum_{\ssource(e)=v} x_e \right)^{\iindeg(v)-1}. \qed \]
\renewcommand{\qedsymbol}{}
\end{proof}

Setting all $x_e=1$ in Theorem~\ref{thm:rootedtreeenum} yields the enumeration
	\begin{equation} \label{rootedproductformula} \kappa(\line G, e_*) = \frac{\kappa(G,w_*)}{\outdeg(v_*)} \pi(G) \end{equation}
where $\kappa(G,w_*)$ is the number of oriented spanning trees of~$G$ rooted at~$w_*$, and 	
	\[ \pi(G) = \prod_{v \in V} \outdeg(v)^{\iindeg(v)-1}. \]  
It is interesting to compare this formula to the theorem of Knuth~\cite{Knuth67}, which in our notation reads
	\begin{equation} \label{knuthproduct} \kappa(\line G, e_*) = \left( \kappa(G,v_*) - \frac{1}{\outdeg(v_*)}\sum_{\substack{\ttarget(e)=v_* \\ e\neq e_*}} \kappa(G,\source(e)) \right) \pi(G). \end{equation}
To see directly why the right sides of (\ref{rootedproductformula}) and (\ref{knuthproduct}) are equal, we define a \emph{unicycle} to be a spanning subgraph of~$G$ which contains a unique directed cycle, and in which every vertex has outdegree~$1$.  If vertex $v_*$ is on the unique cycle of a unicycle~$U$, we say that $U$ goes through $v_*$.

\begin{lemma}
\label{unicycles}
	\[ \kappa^{edge}(G,v_*,\xx) \sum_{\ssource(e)=v_*} x_e = \sum_{\ttarget(e)=v_*} \kappa^{edge}(G,\source(e),\xx) \,x_e. \]
\end{lemma}

\begin{proof}
Removing~$e$ gives a bijection from unicycles containing a fixed edge~$e$ to spanning trees rooted at~$\source(e)$.
If~$U$ is a unicycle through~$v_*$, then the cycle of~$U$ contains a unique edge~$e$ with~$\source(e)=v_*$ and a unique edge~$e'$ with~$\target(e')=v_*$, so both sides are equal to
	\[ \sum_{U} \prod_{e \in U} x_e \]
where the sum is over all unicycles $U$ through $v_*$.
\end{proof}

Setting all $x_e=1$ in Lemma~\ref{unicycles} yields
	\[ \kappa(G,v_*) \,\outdeg(v_*) = \sum_{\ttarget(e)=v_*} \kappa(G,\source(e)). \]
Hence the factor appearing in front of $\pi(G)$ in Knuth's formula (\ref{knuthproduct}) is equal to $\kappa(G,w_*)/\outdeg(v_*)$. 

We conclude this section by discussing some special cases and interesting examples of Theorem~\ref{thm:weightedtreeenum}.

	
\subsection{Deletion and contraction} Fix an edge $e \in E$ which is not a loop, i.e., $\source(e)\neq \target(e)$.  Let 
	\[ G \setminus e = (V,E-\{e\}) \] 
be the graph obtained by deleting~$e$ from~$G$.  While there is more than one sensible way to define contraction for directed graphs, the following definition is natural from the point of view of oriented spanning trees.  Let $G/e$ be the graph obtained from~$G$ by first deleting all edges~$f$ with $\source(f)=\source(e)$, and then identifying the vertices~$\source(e)$ and~$\target(e)$.  Formally, $G/e = (V/e, E/e)$, where
	\[ V/e = V - \{\source(e),\target(e)\} \cup \{e\} \]
and
	\[ E/e = E - \{f | \source(f)=\source(e)\}. \] 
The source and target maps for $G/e$ are given by $p \circ \source \circ i$ and $p \circ \target \circ i$, where $i : E/e \to E$ is inclusion, and $p : V \to V/e$ is given by $p(\source(e))=p(\target(e))=e$, and $p(v)=v$ for $v \neq \source(e),\target(e)$.

With these definitions, the spanning tree enumerator $\kappa^{edge}$ satisfies the following deletion-contraction recurrence.

\begin{lemma}
\label{deletioncontraction}
Let $G$ be a finite directed graph, and let~$e$ be a non-loop edge of~$G$.  Then
	\[ \kappa^{edge}(G,\xx) = \kappa^{edge}(G \setminus e, \xx) + x_e \kappa^{edge}(G/e, \xx).  \]
\end{lemma}

\begin{proof}
Oriented spanning trees of $G \setminus e$ are in bijection with oriented spanning trees of~$G$ that do not contain the edge~$e$.  With the above definition of $G/e$, one easily checks that the map $T \mapsto T \cup \{e\}$ defines a bijection from oriented spanning trees of $G/e$ to oriented spanning trees of $G$ that contain the edge $e$.
\end{proof}

Suppose now that we set $x_f=1$ for all $f\neq e$.  The coefficient of $x_e^\ell$ in $\kappa^{vertex}(\line G,\xx)$ then counts the number of oriented spanning trees~$T$ of~$\line G$ with $\indeg_T(e)=\ell$.  If $v=\source(e)$ has indegree~$k$ and outdegree~$m$, then by Theorem~\ref{thm:weightedtreeenum} and Lemma~\ref{deletioncontraction}, this number is given by the coefficient of $x_e^\ell$ in the product
	\[ \left[ \kappa(G\setminus e) + x_e \kappa(G/e) \right] (m-1+x_e)^{k-1} \prod_{w \neq v} \outdeg(w)^{\indeg(w)-1}. \]
Using the binomial theorem, we obtain the following.

\begin{prop}
Let $G=(V,E)$ be a finite directed graph with no sources.  Fix a non-loop edge $e \in E$ and an integer $\ell \geq 0$.  The number of oriented spanning trees~$T$ of~$\line G$ satisfying $\indeg_T(e)=\ell$ is given by
	\begin{align*} \prod_{w\neq v} \outdeg(w)^{\indeg(w)-1} \left( \binomial{k-1}{\ell} \kappa(G \setminus e) (m-1)^{k-1-\ell} \,+ \right. \\
	+ \left. \binomial{k-1}{\ell-1} \kappa(G/e) (m-1)^{k-\ell} \right)
	\end{align*} 
where $v=\source(e)$, $k = \indeg(v)$ and $m=\outdeg(v)$.
\end{prop}
	
\subsection{Complete graph} Taking $G$ to be the graph with one vertex and $n$ loops, so that $\line G$ is the complete directed graph $\vec{K}_n$ on~$n$ vertices (including a loop at each vertex), we obtain from Theorem~\ref{thm:weightedtreeenum} the classical formula
	\[ \kappa^{vertex}(\vec{K}_n) = (x_1 + \ldots + x_n)^{n-1}. \]
For a generalization to forests, see \cite[Theorem 5.3.4]{Stanley}.
Note that oriented spanning trees of $\vec{K}_n$ are in bijection with rooted spanning trees of the complete undirected graph $K_n$, by forgetting orientation.  
	
\subsection{Complete bipartite graph} Taking $G$ to have two vertices, $a$ and $b$, with $m$ edges directed from $a$ to $b$ and $n$ edges directed from $b$ to $a$, we obtain from Theorem~\ref{thm:weightedtreeenum}
	\[ \begin{split} \kappa^{vertex}(\vec{K}_{m,n}) = 
	(x_1 + \ldots + x_m + y_1 + \ldots + y_n)  \,\times \quad \\ \times\, 
	(x_1 + \ldots + x_m)^{n-1} (y_1 + \ldots + y_n)^{m-1}. \end{split} \]
where $\vec{K}_{m,n} = \line G$ is the bidirected complete bipartite graph on $m+n$ vertices.  The variables $x_1,\ldots,x_m$ correspond to vertices in the first part, and $y_1,\ldots,y_n$ correspond to vertices in the second part.  As with the complete graph, oriented spanning trees of~$\vec{K}_{m,n}$ are in bijection with rooted spanning trees of the undirected complete bipartite graph $K_{m,n}$ by forgetting orientation.


\subsection{De Bruijn graphs} The spanning tree enumerators for the first few de Bruijn graphs are
	\begin{align*} \kappa^{vertex}(DB_1) = x_0 + x_1; \end{align*}
	\begin{align*} \kappa^{vertex}(DB_2) = (x_{00}+x_{01})(x_{10}+x_{11})(x_{01} + x_{10}); \end{align*}
	\begin{align*} \kappa^{vertex}(DB_3) = (x_{000}+x_{001}) (x_{010} + x_{011}) (x_{100} + x_{101}) (x_{110} + x_{111}) \times \\ 
	\times \big(x_{011}x_{110}x_{100} + x_{010}x_{110}x_{100} + x_{110}x_{101}x_{001} + x_{110}x_{100}x_{001} \, + \\
	+\, x_{100}x_{001}x_{011} + x_{101}x_{001}x_{011} + x_{001}x_{010}x_{110} + x_{001}x_{011}x_{110}\big).
		\end{align*}

\section{Sandpile Groups}
\label{sandpilegroups}

Let $G=(V,E)$ be a strongly connected finite directed graph, loops and multiple edges allowed.
Consider the free abelian group $\Z^V$ generated by the vertices of $G$; we think of its elements as formal linear combinations of vertices with integer coefficients.  For $v \in V$ let
	\[ \Delta_v = \sum_{\ssource(e)=v} (\target(e) - v) \in \Z^V \]
where the sum is over all edges $e \in E$ such that $\source(e)=v$.  Fixing a vertex~$v_* \in V$, let $L_V$ be the subgroup of $\Z^V$ generated by $v_*$ and $\{\Delta_v\}_{v\neq v_*}$.  The \emph{sandpile group} $\K(G,v_*)$ is defined as the quotient group
	\[ \K(G,v_*) = \Z^V / L_V. \]
	
The $V \times V$ integer matrix whose column vectors are $\{\Delta_v\}_{v \in V}$ is called the \emph{Laplacian} of $G$.  By Theorem~\ref{thm:rootedmatrixtree}, its principal minor omitting the row and column corresponding to~$v_*$ counts the number $\kappa(G,v_*)$ of oriented spanning trees of~$G$ rooted at $v_*$. 
Since this minor is also the index of $L_V$ in $\Z^V$, we have
	\[ \# \K(G,v_*) = \kappa(G,v_*). \]
Recall that $G$ is \emph{Eulerian} if $\indeg(v)=\outdeg(v)$ for every vertex $v$.
If~$G$ is Eulerian, then the groups $\K(G,v_*)$ and $\K(G,v'_*)$ are isomorphic for any vertices $v_*$ and $v'_*$ \cite[Lemma~4.12]{HLMPPW}.  In this case we usually denote the sandpile group just by~$\K(G)$.  

The sandpile group arose independently in several fields, including arithmetic geometry~\cite{Lor89,Lor91}, statistical physics~\cite{Dhar} and algebraic combinatorics~\cite{Biggs}.  Often it is defined for an undirected graph $G$; to translate this definition into the present setting of directed graphs, replace each undirected edge by a pair of directed edges oriented in opposite directions.  Sandpiles on directed graphs were first studied in \cite{Speer}.  For a survey of the basic properties of sandpile groups of directed graphs and their proofs, see \cite{HLMPPW}.

The goal of this section is to relate the sandpile groups of an Eulerian graph $G$ and its directed line graph $\line G$.  To that end, let $\Z^E$ be the free abelian group generated by the edges of $G$.  For $e \in E$ let
	\[ \Delta_e = \sum_{\ssource(f)=\target(e)} (f-e) \in \Z^E. \]
Fix an edge $e_* \in E$, and let $v_*=\target(e_*)$.  Let $L_E \subset \Z^E$ be the subgroup generated by $e_*$ and $\{\Delta_e\}_{e \neq e_*}$.  Then the sandpile group associated to~$\line G$ and $e_*$ is
 	\[ \K(\line G, e_*) = \Z^E/L_E. \]
Note that $\line G$ may not be Eulerian even when $G$ is Eulerian.  For example, if $G$ is a bidirected graph (i.e., a directed graph obtained by replacing each edge of an undirected graph by a pair of oppositely oriented directed edges) then $G$ is Eulerian, but $\line G$ is not Eulerian unless all vertices of~$G$ have the same degree.

We will work with maps $\phi$ and $\psi$ relating the sandpile groups of~$G$ and~$\line G$.
These maps are analogous to the incidence matrices~$A$ and~$B$ from section~\ref{spanningtrees}, except that now we work over~$\Z$ instead of the field~$\Q(\xx)$.

\begin{lemma}
\label{phidescends}
Let $\phi : \Z^E \to \Z^V$ be the $\Z$-linear map sending $e \mapsto \target(e)$. 
If $G$ is Eulerian, then $\phi$ descends to a surjective group homomorphism 
	\[ \bar{\phi} : \K(\line G, e_*) \to \K(G, v_*). \]
\end{lemma}

\begin{proof}
To show that $\phi$ descends, it suffices to show that $\phi(L_E) \subset L_V$.  For any $e \in E$, we have
	\begin{align*} \phi(\Delta_e) &=  
		\sum_{\ssource(f)=\target(e)} (\target(f)-\target(e))
					= \Delta_{\ttarget(e)}.	
   	\end{align*}
The right side lies in $L_V$ by definition if $\target(e)\neq v_*$.  Moreover, since $G$ is Eulerian,
	\[ \sum_{v \in V} \Delta_v = \sum_{e \in E} (\target(e)-\source(e)) = \sum_{v \in V} (\indeg(v)-\outdeg(v)) v = 0, \]
so $\Delta_{v_*} = - \sum_{v\neq v_*} \Delta_v$ also lies in $L_V$.
Finally, $\phi(e_*) = v_* \in L_V$, and hence $\phi(L_E) \subset L_V$.

Since $G$ is strongly connected, every vertex has at least one incoming edge, so $\phi$ is surjective, and hence $\bar{\phi}$ is surjective.
\end{proof}

Let $k$ be a positive integer.  We say that $G$ is \emph{balanced $k$-regular} if $\indeg(v)=\outdeg(v)=k$ for every vertex~$v$.  Note that any balanced $k$-regular graph is Eulerian; and if $G$ is balanced $k$-regular, then its directed line graph $\line G$ is also balanced $k$-regular.  In particular, this implies
	\[ \sum_{e \in E} \Delta_e = 0 \]
so that $\Delta_{e_*} \in L_E$.

Now consider the $\Z$-linear map 
	 \begin{equation*} \begin{split} 
	 	\psi : \Z^V \to \Z^E \\
	\end{split} \end{equation*} 
sending $v \mapsto \sum_{\ssource(e)=v} e$.  For a group $\Gamma$, write $k\Gamma = \{ kg | g \in \Gamma\}$.

\begin{lemma}
\label{multiplesofk}
If $G$ is balanced $k$-regular, then $\psi$ descends to a group isomorphism 
	\[ \bar{\psi} : \K(G) \xrightarrow{\simeq} k\,\K(\line G). \]
\end{lemma}

\begin{proof}
We have
	\[ \psi(v_*) = \Delta_{e_*} + k e_* \in L_E \]
and for any vertex $v \in V$,
	\begin{align*}
	\psi(\Delta_v) &= \sum_{\ssource(e)=v} \psi(\target(e)) - k\psi(v) \\
			&=  \sum_{\ssource(e)=v} \sum_{\ssource(f)=\ttarget(e)} f - k\sum_{\ssource(g)=v} g \\
			&=  \sum_{\ssource(e)=v} \left( \sum_{\ssource(f)=\ttarget(e)} f - ke \right) \\
			&=  \sum_{\ssource(e)=v} \Delta_e.
	\end{align*}
Since $\line G$ is Eulerian, the right side lies in $L_E$.  Hence $\psi(L_V) \subset L_E$, and $\psi$ descends to a group homomorphism
	\[ \bar{\psi} : \K(G) \to \K(\line G). \]
If $v$ is any vertex of $G$, and $e$ is any edge with $\target(e)=v$, then
	\[ \psi(v) = ke + \Delta_{e}, \]
so the image of $\bar{\psi}$ is $k\,\K(\line G)$.  

To complete the proof it suffices to show that  $\psi^{-1}(L_E) \subset L_V$, so that $\bar{\psi}$ is injective.  If $k=1$ then $K(G)$ is the trivial group, so there is nothing to prove.  Assume now that $k\geq 2$. Given $\eta \in \Z^V$ with $\psi(\eta) \in L_E$, write
	\[ \psi(\eta) = \sum_{e\in E} b_e \Delta_e + b_*e_* \]
for some coefficients $b_e, b_* \in \Z$.  Then
	\begin{align*} \psi(\eta) - b_*e_* 
	&= \sum_{e\in E} b_e \left( \sum_{\ssource(f)=\ttarget(e)} f - ke \right) \\
	&= \sum_{f \in E} \left( \sum_{\ttarget(e)=\ssource(f)} b_e \right) f - \sum_{e \in E} kb_e e  \\
	&= \sum_{f \in E} \left( \sum_{\ttarget(e)=\ssource(f)} b_e - kb_f \right) f.
	\end{align*}					
Now writing $\eta = \sum_{v \in V} a_v v$, so that $\psi(\eta) = \sum_{f \in E} a_{\ssource(f)} f$, equating coefficients of $f$ gives
	\begin{equation} \label{equatingedgecoefs} kb_f = \sum_{\ttarget(e)=\ssource(f)} b_e - a_{\ssource(f)}, \qquad f\neq e_*. \end{equation}
Note that the right side depends only on $\source(f)$.  For $v \in V$, let
	\[ F(v) = \frac{1}{k} \sum_{\ttarget(e)=v} b_e - \frac{1}{k} a_v. \]  
Then $b_f = F(\source(f))$ for all edges $f \neq e_*$.  Since $k \geq 2$, for any $v \in V$ there exists an edge $f \neq e_*$ with $\source(f)=v$.  Moreover if $v \neq v_*$ and $\target(e)=v$, then $e\neq e_*$.  From (\ref{equatingedgecoefs}) we obtain
	\[ a_v = \sum_{\ttarget(e)=v} b_e - kb_f = \sum_{\ttarget(e)=v} F(\source(e)) - k F(v), \qquad v\neq v_*. \]
Hence	
	\begin{align*}
	\eta - a_{v_*}v_* = \sum_{v \neq v_*} a_v v 
		&= \sum_{e\in E, ~\ttarget(e)\neq v_*} F(\source(e)) \target(e) - \sum_{v\neq v_*} k F(v) v \\
		&= \sum_{v \in V} F(v) \left( \sum_{\ssource(e)=v, ~\target(e)\neq v_*} \target(e) - k v \right) + kF(v_*)v_* \\
		&= \sum_{v \in V} F(v) \Delta_v + \left(kF(v_*) - \sum_{\ttarget(e)=v_*} F(\source(e)) \right) v_*.
	\end{align*}
The right side lies in $L_V$, so $\eta \in L_V$, completing the proof.
\end{proof}

\begin{proof}[Proof of Theorem~\ref{mainsequence}]
If $G$ is Eulerian, then $\phi$ descends to a surjective homomorphism of sandpile groups by Lemma~\ref{phidescends}.  If $G$ is balanced $k$-regular, then $\bar{\psi}$ is injective by Lemma~\ref{multiplesofk}, so
	\[ \ker(\bar{\phi}) = \ker(\bar{\psi} \circ \bar{\phi}). \]
Moreover for any edge $e \in E$
	\[ (\psi \circ \phi)(e) = \sum_{\ssource(f)=\target(e)} f = ke + \Delta_e. \]
Hence $\bar{\psi} \circ \bar{\phi}$ is multiplication by $k$, and $\ker(\bar{\phi})$ is the $k$-torsion subgroup of $\K(\line G)$.
\end{proof}

\section{Iterated Line Graphs}
\label{iteratedlinegraphs}

Let $G=(V,E)$ be a finite directed graph, loops and multiple edges allowed.
The \emph{iterated line digraph} $\line^n G = (E_n,E_{n+1})$ has as vertices the set
	\[ E_n = \{ (e_1,\ldots,e_n) \in E^n \,|\, \source(e_{i+1})=\target(e_i), ~i=1,\ldots,n-1 \} \]
of directed paths of $n$ edges in $G$.  The edge set of $\line^n G$ is $E_{n+1}$, and the incidence is defined by
	\begin{align*} \source(e_1,\ldots,e_{n+1}) &= (e_1,\ldots,e_n); \\
	 \target(e_1,\ldots,e_{n+1}) &= (e_2,\ldots,e_{n+1}). \end{align*}
(We also set $E_0=V$, and $\line^0 G = G$.)
For example, the de~Bruijn graph~$DB_n$ is~$\line^n(DB_0)$, where~$DB_0$ is the graph with one vertex and two loops.  

Our next result relates the number of spanning trees of $G$ and $\line^n G$.  Given a vertex $v \in V$, let
	\[ p(n,v) = \# \{ (e_1,\ldots,e_n)\in E_n \,|\, \target(e_n)=v \} \]
be the number of directed paths of $n$ edges in $G$ ending at vertex $v$.  
\begin{theorem}
\label{iterated}
Let $G=(V,E)$ be a finite directed graph with no sources.  Then
	\[ \kappa(\line^n G) = \kappa(G) \prod_{v \in V} \mbox{\em outdeg}(v)^{p(n,v)-1}. \]
\end{theorem}

\begin{proof}
For any $j \geq 0$, by Theorem~\ref{thm:weightedtreeenum} applied to $\line^{j} G$ with all edge weights~$1$,
	\begin{align*} \frac{\kappa(\line^{j+1} G)}{\kappa(\line^{j} G)}  &= \prod_{(e_1,\ldots,e_{j}) \in E_{j}} \outdeg(\target(e_{j}))^{\iindeg(\ssource(e_1))-1} \\
			&= \prod_{v \in V} \outdeg(v)^{p(j+1,v)-p(j,v)}.
	\end{align*}
Taking the product over $j=0,\ldots,n-1$ yields the result.
\end{proof}

When $G$ is balanced $k$-regular, we have $p(n,v)=k^n$ for all vertices~$v$, so we obtain as a special case of Theorem~\ref{iterated} the result of Huaxiao, Fuji and Qiongxiang \cite[Theorem~1]{iterated}
	\[ \kappa(\line^n G) = \kappa(G) k^{(k^n -1)\#V}. \]
In particular, taking $G=DB_0$ yields the classical formula
	\[ \kappa(DB_n) = 2^{2^n-1}. \]
Since $DB_n$ is Eulerian, the number $\kappa(DB_n,v_*)$ of oriented spanning trees rooted at~$v_*$ does not depend on $v_*$, so
	\begin{equation} \label{twotothenminusnminusone} \kappa(DB_n,v_*)= 2^{-n}\kappa(DB_n) = 2^{2^n-n-1}. \end{equation}
This familiar number counts \emph{de Bruijn sequences} of order~$n+1$ (Eulerian tours of $DB_n$) up to cyclic equivalence.  De Bruijn sequences are in bijection with oriented spanning trees of $DB_n$ rooted at a fixed vertex~$v_*$; for more on the connection between spanning trees and Eulerian tours, see \cite{EdB51} and \cite[section 5.6]{Stanley}.

Perhaps less familiar is the situation when $G$ is not regular.  As an example, consider the graph
	\[ G = (\{0,1\},\{(0,0),(0,1),(1,0)\}). \]
The vertices of its iterated line graph $\line^n G$ are binary words of length $n+1$ containing no two consecutive~$1$'s.  The number of such words is the Fibonacci number~$F_{n+3}$, and the number of words ending in~$0$ is~$F_{n+2}$.  By Theorem~\ref{iterated}, the number of oriented spanning trees of~$\line^n G$ is
	\[ \kappa(\line^n G) 
	= 2 \cdot 2^{p(n,0)-1} 
	= 2^{F_{n+2}}. \]

Next we turn to the proofs of Theorems~\ref{deBruijn} and~\ref{Kautz}.  If~$a$ and~$b$ are positive integers, we write $\Z_b^a$ for the group $(\Z/b\Z) \oplus \ldots \oplus (\Z/b\Z)$ with~$a$ summands.

\begin{proof}[Proof of Theorem~\ref{deBruijn}]
Induct on $n$.
From (\ref{twotothenminusnminusone}) we have 
	\[ \# \K(DB_n) = 2^{2^n -n-1} \]
hence
	\[ \K(DB_n) = \Z_2^{a_1} \oplus \Z_4^{a_2} \oplus \Z_8^{a_3} \oplus \ldots \oplus \Z_{2^m}^{a_m} \]
for some nonnegative integers $m$ and $a_1, \ldots, a_m$ satisfying
	\begin{equation} \label{totalexponent} \sum_{j=1}^m j a_j = 2^n - n -1. \end{equation}
By Lemma~\ref{multiplesofk} and the inductive hypothesis,
	\begin{align*} 
	\Z_2^{a_2} \oplus \Z_4^{a_3} \oplus \ldots \oplus \Z_{2^{m-1}}^{a_m} &\simeq 2 \K(DB_n) \\ &\simeq  \K(DB_{n-1}) \\
		&\simeq \Z_2^{2^{n-3}} \oplus \Z_4^{2^{n-4}} \oplus \ldots \oplus \Z_{2^{n-2}}.
	\end{align*}
hence $m=n-1$ and
	\[ a_2 = 2^{n-3}, a_3 = 2^{n-4}, \ldots, a_{n-1}=1. \]
Solving (\ref{totalexponent}) for $a_1$ now yields $a_1=2^{n-2}$.
\end{proof}

For $p$ prime, by carrying out the same argument on a general balanced $p$-regular directed graph~$G$ on~$N$ vertices, we find that
	\[ \K(\line^n G) \simeq \tilde{K} \oplus \bigoplus_{j=1}^{n-1} (\Z_{p^j})^{p^{n-1-j}(p-1)^2 N} \oplus (\Z_{p^n})^{(p-1)N-r-1} \oplus \bigoplus_{j=1}^{m} (\Z_{p^{n+j}})^{a_j} \]
where
	\[ \Sylow_p(\K(G)) = (\Z_p)^{a_1} \oplus \ldots \oplus (\Z_{p^m})^{a_m}; \]
	\[ \tilde{K} = \K(G) / \Sylow_p(\K(G)); \]
	\[ r = a_1 + \ldots + a_m. \]
In particular, taking $G=\Kautz_1$ with $p=2$, we have $\K(G)=\tilde{K}=\Z_3$, and we arrive at Theorem~\ref{Kautz}.

\section{Concluding Remarks}
\label{concluding}

Theorem~\ref{mainsequence} describes a map from the sandpile group $\K(\line G, e_*)$ to the group $\K(G,v_*)$ when $G$ is an Eulerian directed graph and $e_*=(w_*,v_*)$ is an edge of~$G$.  
There is also a suggestive numerical relationship between the orders of the sandpile groups $\K(\line G, e_*)$ and $\K(G,w_*)$, which holds even when $G$ is not Eulerian: by equation (\ref{rootedproductformula}) we have
	\[ \kappa(G,w_*) \,|\, \kappa(\line G, e_*) \]
whenever $G$ satisfies the hypothesis of Theorem~\ref{thm:rootedtreeenum}.  
This observation leads us to ask whether $\K(G,w_*)$ can be expressed as a subgroup or quotient group of $\K(\line G, e_*)$.  

The area of spanning trees, Eulerian tours, and sandpile groups is full of simple enumerative results with no known bijective proofs.  To give just one example, the number of de Bruijn sequences of order~$n$ (Eulerian tours of~$DB_{n-1}$) with distinguished starting edge is~$2^{2^{n-1}}$.  Richard Stanley has posed the problem of finding a bijection between ordered pairs of such sequences and all~$2^{2^n}$ binary words of length~$2^n$.  This problem and a number of others could be solved by giving a bijective proof of Theorem~\ref{thm:weightedtreeenum}.  

\end{document}